\documentclass{amsart}

\usepackage{graphicx}

\newtheorem{theorem}{Theorem}[section]
\newtheorem{lemma}[theorem]{Lemma}
\newtheorem{proposition}[theorem]{Proposition}

\theoremstyle{definition}

\theoremstyle{remark}
\newtheorem{remark}[theorem]{Remark}

\numberwithin{equation}{section}

\usepackage{ulem}
\usepackage{amsrefs}

\usepackage{color}

\let\oldmarginpar\marginpar
\renewcommand\marginpar[1]{\-\oldmarginpar[\raggedleft\footnotesize #1]
{\raggedright\footnotesize #1}}

\usepackage[T1]{fontenc}
\usepackage{lmodern}
\usepackage[latin1]{inputenc}

\usepackage{setspace}
\usepackage{indentfirst}

\usepackage{multicol}
\usepackage{geometry}
\usepackage{hyperref}
\usepackage{amsmath,amssymb}
\usepackage{mathrsfs}
\usepackage{amsthm}

\theoremstyle{plain}

\newtheorem*{theorem*}{Theorem}
\newtheorem*{corollary*}{Corollary}

\newcommand{\cE}{\mathcal{E}}

\newcommand{\C}{\mathbb{C}}
\newcommand{\CP}{\mathbb{C}\mathbf{P}}
\newcommand{\R}{\mathbb{R}}
\newcommand{\Z}{\mathbb{Z}}

\renewcommand{\emptyset}{\varnothing}
\newcommand{\Lie}{\text{Lie}}

\DeclareMathOperator{\Hom}{Hom}

\begin{document}

\title{Holomorphic polyvector fields on toric varieties}

\author{Wei Hong}
\address{School of Mathematics and Statistics, Wuhan University, Wuhan, 430072, China}
\address{Hubei Key Laboratory of Computational Science, Wuhan University, Wuhan, 430072, China}
\email{hong\textunderscore  w@whu.edu.cn}


\begin{abstract}
In this paper, we give an explicit description of holomorphic polyvector fields on  smooth compact toric varieties, which generalizes Demazure's result of holomorphic vector fields on toric varieties. 
\end{abstract}

\maketitle


\section{introduction}\label{sect-Intr}

Polyvector fields appeared in the recent work of many mathematicians. 
Barannikov and Kontsevich \cite{B-K 98} shows that polyvector fields plays important roles in deformation theory and mirror symmetry. 
In \cite{D-R-W 15}, Dolgushev, Rogers and Willwacher gives a further study of deformation theory related to polyvector fields. 
In differential geomtry, Hitchin  \cite{Hitchin 11} study some problems related to holomorphic polyvector fields.  Polyvector fields also appeared naturally in generalized complex geometry according to Gualtieri's work \cite{Gualtieri 11}.

This paper is devoted to the study of holomorphic polyvector fields on toric varieties.
In \cite{Demazure}, Demazure described all the holomorphic vector fields on smooth compact toric varieties. In this paper, we give an explicit description of all holomorphic polyvector fields on smooth compact toric varieties in Theorem \nameref{General-multiVect-thm}, which generalizes Demazure's results of holomorphic vector fields on toric varieties.

Recall that a toric variety \cite{Cox} is an irreducible variety $X$ such that
\begin{enumerate}
\item $T=(\C^*)^n$ is a Zariski open set of $X$, 
\item the action of $T=(\C^*)^n$ on itself extends to an action of $T=(\C^*)^n$ on $X$.
\end{enumerate}
One may consult \cite{Oda} and \cite{Fulton} for more details of toric varieties.

Let $N=\Hom(\C^*,T)\cong\Z^n$ and $M=\Hom(T, \C^*)$.
Then $M\cong \Hom_\Z(N,\Z)$ and $N\cong \Hom_\Z(M,\Z)$.
Each element $m$ in $M$ gives rise to a character $\chi^m\in Hom(T,\C^*)$,
which can also be considered as a rational function on $X$. 
Let $N_{\R}=N\otimes_{\Z}\R\cong\R^n$. Let $M_{\R}=M\otimes_{\Z}\R\cong\R^n$ 
be the dual space of $N_{\R}$.


Let $X=X_\Delta$ be a smooth compact toric variety associated with a fan $\Delta$ in $N_\R\cong\R^n$. 
The $T$-action on $X$ induces a $T$-action on 
$H^0(X, \wedge^k T_{X})$, the vector space of holomorphic $k$-vector fields on $X$. 
Denote by $V_I^k$ the weight space corresponding to the character $I\in M$. 
We have
\begin{equation}\label{VIk-eqn}
H^0(X, \wedge^kT_{X})=\bigoplus_{I\in M}V_I^k.
\end{equation}
Since $H^0(X, \wedge^kT_{X})$ is a finite dimensional vector space, 
there are only finite elements $I\in M$ such that $V_I^k\neq 0$.

Next we give an explicit description of the vector spaces $V_I^k$ for smooth compact toric varieties. 

\begin{itemize}
\item
Let $N_\C=N\otimes_{\Z}\C$. 
Since $T\cong N\otimes_\Z\C^*$, we have $\Lie(T)\cong N_\C$, 
where $\Lie(T)$ denotes the Lie algebra of $T$. 
We define a map $\rho:N_\C\rightarrow \mathfrak{X}(X)$ by
\begin{equation} \label{rho-def-eqn}
\rho:N_\C=N\otimes_\Z\C\cong \Lie(T)\rightarrow \mathfrak{X}(X),
\end{equation}
where $\Lie(T)\rightarrow \mathfrak{X}(X)$ is defined by the infinitesimal action of the Lie algebra $\Lie(T)$ on $X$. 
And the map $\rho:N_\C\rightarrow \mathfrak{X}(X)$ is determined by
$$\rho(x)(\chi^m)=\langle x,m\rangle\chi^m$$
for any $x\in N_\C$ and $m\in M$.
Since the action map $T\times X\rightarrow X$ is holomorphic, 
the images of $\rho$ are holomorphic vector fields on $X$.
By abuse of notation, the induced maps
\begin{equation} \label{rho-def-eqn2}
\wedge^k N_\C\rightarrow\mathfrak{X}^k(X)
\end{equation}
 are also denoted by $\rho$ for $2\leq k\leq n$.
 Let $W=\rho(N_\C)$, $W^k=\wedge^k W$ and $W^0=\C$. 
Then $W^k$ is a subspace of $H^0(X, \wedge^kT_{X})$. 
And the map $\wedge^k N_\C\xrightarrow{\rho} W^k$ is an isomorphism.

\item
Let $\Delta(k)$ $(0\leq k\leq n)$ the set of all $k$-dimensional cones in 
$\Delta$.
Suppose that $\alpha_t ~(1\leq t\leq r)$ are all one dimensional cones in $\Delta(1)$. 
Let  $e(\alpha_t)\in N$ $(1\leq t\leq r)$ be the corresponding primitive elements, i.e., the unique generator of $\alpha_t \cap N$. 
Let $P_\Delta$ be a polytope in $M_\R$ defined by
\begin{equation*}
P_\Delta=\bigcap_{\alpha_t\in \Delta(1)}\{I\in M_\R\mid \langle I,e(\alpha_t)\rangle\geq -1\}.
\end{equation*}
Since $X_{\Delta}$ is compact, $P_{\Delta}$ is a compact polytope in $M_{\R}$. 
Let 
\begin{equation*}
S_\Delta=\{I\mid I\in M\cap P_{\Delta}\}.
\end{equation*}
Then $S_{\Delta}$ is a non-empty finite set and $0\in S_{\Delta}$.
We denote by $P_\Delta(i)$ the set of all $i$-dimensional faces of the polytope $P_\Delta$.  Let
\begin{equation*}
S(\Delta,i)=\bigcup_{F_j\in P_\Delta(n-i)}\{I\in int(F_j)\cap M\}
\end{equation*}
 for $0\leq i\leq n$, where $int(F_j)$ denotes the relative interior of the face $F_j$.
Let
\begin{equation*}
 S_k(\Delta)=\bigcup_{0\leq i\leq k}S(\Delta,i)
 \end{equation*}
  for all $0\leq k\leq n$. 
  
\item  
For any $I\in S_{\Delta}$, there exists a unique face $F_I(\Delta)$ of $P_{\Delta}$, 
such that $I\in int(F_I(\Delta))\cap M$. 
Let
\begin{equation*}
E_I(\Delta)=\{e(\alpha_t)\mid\alpha_t \in\Delta(1)~\text{and}~\langle I,e(\alpha_t)\rangle=-1\}.
\end{equation*} 
Suppose that $E_I(\Delta)=\{e(\alpha_{s_1}),\ldots e(\alpha_{s_l})\}$, 
where $1\leq s_1<\ldots<s_l\leq r$.
Let $F^{\perp}_I(\Delta)\subseteq N_{\R}$ be the normal space of $F_I(\Delta)$,
which is defined by
\begin{equation*}
F^{\perp}_I(\Delta)=\sum_{1\leq t\leq l} \R\cdot e(\alpha_{s_t}).
\end{equation*}
And we set
\begin{equation*} 
F^{\perp}_I(\Delta)=0 \quad\text{if}\quad E_I(\Delta)=\emptyset.
\end{equation*}
Suppose that $I\in S(\Delta,i)$. 
Then we have $\dim F_I(\Delta)=n-i$ and $\dim F^{\perp}_I(\Delta)=i$.

\item
We choose a nonzero element $\mathcal{E}_I(\Delta)\in\wedge^i F^{\perp}_I(\Delta)\cong\R$.
Since $F^{\perp}_I(\Delta)\subseteq N_{\R}\subseteq N_{\C}$,
$\mathcal{E}_I(\Delta)$ can be considered as an element in $\wedge^i N_{\C}$.
Let $\mathcal{V}_I(\Delta)=\rho(\mathcal{E_I}(\Delta))\in W^i$ be an $i$-vector field on $X=X_{\Delta}$.
Let
\begin{equation*}
W_I^k(\Delta)=
\begin{cases}
\C\cdot\mathcal{V}_I(\Delta)\wedge  W^{k-i}\quad &\text{for}~ 0<i\leq k,\\
W^k&\text{for}~ i=0,\\
0\quad &\text{for}~i>k.
\end{cases}
\end{equation*}
Then $W_I^k(\Delta)$ is a subspace of $W^k$.
We define
\begin{equation*}
V_I^k(\Delta)=\chi^I\cdot W_I^k(\Delta)=\{\chi^I\cdot w\mid w\in W_I^k(\Delta)\}.
\end{equation*}
The elements in $V_I^k(\Delta)$ are considered as meromorphic $k$-vector fields on $X=X_{\Delta}$. In section \ref{sect-multi}, we prove that for any $I\in S_k(\Delta)$, the
$k$-vector fields in $V_I^k(\Delta)$ have moveable singularities. And we use them to represent the corresponding holomorphic $k$-vector fields in this paper.
\end{itemize}

The next theorem gives a description of the holomorphic polyvector fields on toric varieties. 
\begin{theorem*}[A]\label{General-multiVect-thm}
Let $X=X_ \Delta$ be a smooth compact toric variety associated with a fan $ \Delta$ in $N_\R$. 
\begin{enumerate}
\item  We have the following decomposition
\begin{equation}\label{General-multiVect-thm-eqn1}
H^0(X,\wedge^kT_{X})=\bigoplus_{I\in S_k(\Delta)}V_I^k(\Delta)
\end{equation}
for all $0\leq k\leq n$, where $V_I^k(\Delta)=V_{-I}^k$ for all $I\in S_k(\Delta)$.
\item 
The dimension of holomorphic polyvector fields on $X=X_\Delta$ can be computed by
\begin{equation}\label{General-multiVect-thm-eqn2}
\dim H^0(X,\wedge^kT_{X})
=\sum_{0\leq i\leq k}~\sum_{F_j\in P_{\Delta}( n-i)}{{n-i}\choose{k-i}}\cdot\#(int(F_j)\cap M)
\end{equation}
for all $0\leq k\leq n$.
\end{enumerate}
\end{theorem*}

We give some comments for Theorem \nameref{General-multiVect-thm}:
\begin{itemize}
\item In the case of $k=1$, the above theorem was proved by Demazure \cite{Demazure}. 
And the set $$R(N, \Delta)=S_1(\Delta)\backslash\{0\}$$ is called the root system for 
$(N, \Delta)$.
\item In the case of $k=n$, we have the well-known results
\begin{gather*}
H^0(X,\wedge^nT_{X})\cong
\bigoplus_{I\in P_{\Delta}\cap M}\C\chi^I\cdot   W^n,\\
\dim H^0(X,\wedge^nT_{X})=\#(P_{\Delta}\cap M).
\end{gather*}
\item The special case of $X=\CP^n$ was proved in \cite{Hong 19}. 
We would like to point out a sign mistake in the Theorem 3.3 of \cite{Hong 19}, where $V_I^k$ should be the weight space corresponding to the character $-I$. 
\item 
If $X$ is a toric Fano manifold, Equation \eqref{General-multiVect-thm-eqn2} can be obtained by Theorem $3.6$ in \cite{Materov 02}  
since $\wedge^k T_X\cong\Omega^{n-k}_X\otimes\wedge^n T_X$.
\end{itemize}

Notice that the Gerstenhaber algebra structure of the polyvector fields on the algebraic torus is studied in  \cite{M-R 19}.  By Theorem \nameref{General-multiVect-thm}, the Gerstenhaber algebra structures of the holomorphic polyvcetor fields on smooth compact toric varieties can be studied in the similar way. But it is more complicated since it is related to the lattices in polytopes. However, those will be left for future works.

The paper is organized in the following way.
In Section \ref{polytope-sect}, we construct some sets and vector spaces related to toric varieties and study their properties.
In Section \ref{sect-multi},  we study holomorphic polyvector fields on toric varieties and prove  the Theorem \nameref{General-multiVect-thm}. 

{\bf Acknowledgements} 
We would like to thank Yu Qiao, Mathieu Sti\'{e}non, Xiang Tang and Ping Xu for helpful discussions and comments. Hong's research is partially supported by NSFC grant 12071241 and NSFC grant 12071358.

\section{Toric varieties, lattices and polytopes}\label{polytope-sect}

\subsection{The polytope $P_{\Delta}$ and the set $S_\Delta$} 
Let $X=X_\Delta$ be a smooth compact toric variety.
Suppose that $\alpha_t ~(1\leq t\leq r)$ are all one dimensional cones in $\Delta(1)$. 
Let  $e(\alpha_t)\in N$ $(1\leq t\leq r)$ be the corresponding primitive elements, i.e., the unique generator of $\alpha_t \cap N$. 
Let
\begin{equation*}
E(\Delta)=\{e(\alpha_1),\ldots, e(\alpha_r)\}.
\end{equation*}

The sets $S_\Delta$, $S(\Delta,i)$ and $S_k(\Delta)$ have been defined in the previous section. And we have
\begin{enumerate}
\item 
$S(\Delta,i)\cap S(\Delta,j)=\emptyset$ for all $0\leq i\neq j\leq n$.
\item
$S_0(\Delta)\subseteq S_1(\Delta)\subseteq\ldots\subseteq S_n(\Delta)=S_{\Delta}$
\item
$S_{k+1}(\Delta)=S_k(\Delta)\cup S(\Delta,k+1).$
\end{enumerate}

\begin{lemma}\label{S0Delta-lem}
Let $X=X_\Delta$ be a smooth compact toric variety of dimension $n$. Then we have
\begin{equation*}
S_0(\Delta)=\{0\}.
\end{equation*}
\end{lemma}
\begin{proof}
For any $I\in S_0(\Delta)$, we have
\begin{equation}\label{S0Delta-lem-eqn1}
\langle I, e(\alpha_t)\rangle\geq 0
\end{equation}
for all $\alpha_t\in \Delta(1)$. 
Since $X$ is compact, we have
\begin{equation*}
\bigcup_{\sigma\in\Delta(n)}\sigma=N_\R.
\end{equation*} 
Therefore for any $x\in N_\R$, there exists $\sigma\in\Delta(n)$, such that $x\in\sigma$.
Since $X$ is smooth, $x\in\sigma$ can be written as
\begin{equation}\label{S0Delta-lem-eqn2}
x=\sum_{\alpha_t\in\Delta(1)\cap\sigma} \lambda_t e(\alpha_t), \quad\lambda_t\geq 0.
\end{equation}
 By Equation  \eqref{S0Delta-lem-eqn1} and Equation \eqref{S0Delta-lem-eqn2}, 
 we have
 \begin{equation*}
\langle I,x\rangle\geq 0
\end{equation*}
for all $x\in N_\R$, which implies $I=0$. 
Hence $S_0(\Delta)=\{0\}.$
\end{proof}

The following results are obviously. However, we need them later and put them as a lemma here. 
\begin{lemma}\label{SDelta-lem}
\begin{enumerate}
\item
For any $I\in S_{\Delta}$, the following statements are equivalent:
\begin{enumerate}
\item $I\in S(\Delta,i)$,
\item $F_I(\Delta)\in P_{\Delta}(n-i)$,
\item $\dim F^{\perp}_I(\Delta)=i$.
\end{enumerate}
\item
For any $I\in S_{\Delta}$, the following statements are equivalent:
\begin{enumerate}
\item $I\in S_k(\Delta)$,
\item $\dim F^{\perp}_I(\Delta)\leq k$.
\end{enumerate}
\end{enumerate}
\end{lemma}

\subsection{The polyhedron $P_{\sigma}$ and the set $S_\sigma$} 
Let $\sigma$ be a smooth cone of dimension $n$. Then $\sigma$ can be written as
$$\sigma=\sum_{t=1}^{n}\R_{\geq 0}\cdot e_t(\sigma),$$
where $\{e_1(\sigma),\ldots, e_n(\sigma)\}$ is a $\Z$-basis of $N$.
Let
\begin{equation*}
 E(\sigma)=\{e_1(\sigma),\ldots, e_n(\sigma)\}.
 \end{equation*}

Let $P_{\sigma}$ be a polyhedron defined by
\begin{equation}\label{Psigma-eqn}
P_\sigma=\{I\in M_\R\mid \langle I, e_t(\sigma)\rangle\geq -1~ \text{for all}~ 1\leq t\leq n\}.
\end{equation}
Let
\begin{equation*}
S_{\sigma}=P_{\sigma}\cap M.
\end{equation*}

We denote by $P_\sigma(i)$ the set of all $i$-dimensional faces of the polyhedron $P_\Delta$.  Let
\begin{equation*}
S(\sigma,i)=\bigcup_{F_j\in P_\sigma(n-i)}\{I\in int(F_j)\cap M\}
\end{equation*}
 for $0\leq i\leq n$, where $int(F_j)$ denotes the relative interior of $F_j$.
 Then we have 
 \begin{equation*}
 S(\sigma,i)\cap S(\sigma,j)=\emptyset
 \end{equation*}
 for all $0\leq i\neq j\leq n$. 
Let
\begin{equation*}
 S_k(\sigma)=\bigcup_{0\leq i\leq k}S(\sigma,i)
 \end{equation*}
  for all $0\leq k\leq n$. 
 Then we have
$$S_0(\sigma)\subseteq S_1(\sigma)\subseteq\ldots\subseteq S_n(\sigma)
=S_{\sigma}.$$

For any $I\in S_{\sigma}$,  there exists a unique face $F_I(\sigma)$ of $P_{\sigma}$, 
such that $I\in int(F_I(\sigma))\cap M$. 
 Let 
$$E_I(\sigma)=\{e_t(\sigma)\in E(\sigma)\mid \langle I,e_t(\sigma)\rangle=-1\}.$$

Suppose that $E_I(\sigma)=\{e_{s_1}(\sigma),\ldots e_{s_j}(\sigma)\}$, 
where $1\leq s_1<\ldots<s_j\leq n$.
Let $F^{\perp}_I(\sigma)\subseteq N_{\R}$ be the normal space of $F_I(\sigma)$, which is defined by defined by 
\begin{equation*}
F^{\perp}_I(\sigma)=\sum_{1\leq t\leq j} \R\cdot e_{s_t}(\sigma).
\end{equation*}
And we define
\begin{equation*} 
F^{\perp}_I(\sigma)=0 \quad\text{if}\quad E_I(\sigma)=\emptyset.
\end{equation*}

Similar to Lemma \ref{SDelta-lem}, we have
\begin{lemma}\label{Ssigma-lem}
\begin{enumerate}
\item
For any $I\in S_{\sigma}$, the following statements are equivalent:
\begin{enumerate}
\item $I\in S(\sigma,j)$,
\item $F_I(\sigma)\in P_{\sigma}(n-j)$,
\item $\dim F^{\perp}_I(\sigma)=j$,
\item $\# E_I(\sigma)=j$.
\end{enumerate}
\item
For any $I\in S_{\sigma}$, the following statements are equivalent:
\begin{enumerate}
\item $I\in S_k(\sigma)$,
\item $\dim F^{\perp}_I(\sigma)\leq k$,
\item $\# E_I(\sigma)\leq k$.
\end{enumerate}
\end{enumerate}
\end{lemma}

Now we have the following proposition.
\begin{proposition}\label{PFS-lem}
Let $X_\Delta$ be a smooth compact toric variety. 
\begin{enumerate}
\item
We have
\begin{equation*}
P_{\Delta}=\bigcap_{\sigma\in\Delta(n)} P_{\sigma}
\end{equation*}
and 
\begin{equation*}
S_{\Delta}=\bigcap_{\sigma\in\Delta(n)} S_{\sigma}.
\end{equation*}
\item
For any $\sigma\in\Delta(n)$ and $I\in P_{\Delta}\subset P_{\sigma}$, 
$F^{\perp}_I(\sigma)$ is a subspace of $F^{\perp}(\Delta)$. 
Moreover, we have
\begin{equation*}
F^{\perp}_I(\Delta)=\sum_{\sigma\in\Delta(n)}F^{\perp}_I(\sigma).
\end{equation*}
\item
We have
\begin{equation*}
S_k(\Delta)\subseteq\bigcap_{\sigma\in\Delta(n)}S_k(\sigma).
\end{equation*}
\end{enumerate}
\end{proposition}

\begin{proof}
\begin{enumerate}
\item
Since $X_\Delta$ is a smooth compact toric variety, we have 
$|\Delta|=\bigcup_{\sigma\in\Delta(n)}|\sigma|$. 
As a consequence, we have 
\begin{equation*}
\bigcup_{\sigma\in\Delta(n)}\{e_1(\sigma),\ldots,e_n(\sigma)\}=\{e(\alpha_1),\ldots,e(\alpha_r)\},
\end{equation*}
which implies
$$P_{\Delta}=\bigcap_{\sigma\in\Delta(n)} P_{\sigma}.$$
Since $S_\Delta=P_\Delta\cap M$ and $S_\sigma=P_\sigma\cap M$, 
we get that
$$S_{\Delta}=\bigcap_{\sigma\in\Delta(n)} S_{\sigma}.$$

\item
Since $E_I(\sigma)\subseteq E_I(\Delta)$,  
we get that $F^{\perp}_I(\sigma)$ is a subspace of $F^{\perp}_I(\Delta)$.  
And by $$\bigcup_{\sigma\in\Delta(n)} E_I(\sigma)=E_I(\Delta),$$
we get that $F^{\perp}_I(\Delta)=\sum_{\sigma\in\Delta(n)}F^{\perp}_I(\sigma)$. 
\item
For any $I\in S_k(\Delta)$, by Lemma \ref{SDelta-lem}, we have
$$\dim F^{\perp}_I(\Delta)\leq k.$$
For any $\sigma\in\Delta(n)$, 
since $F^{\perp}_I(\sigma)$ is a subspace of $F^{\perp}_I(\Delta)$, we have 
 $$\dim F^{\perp}_I(\sigma)\leq\dim F^{\perp}_I(\Delta)\leq k.$$
 By Lemma \ref{Ssigma-lem}, we have $I\in S_k(\sigma)$.
 
 Therefore, we have
 $$S_k(\Delta)\subseteq S_k(\sigma)$$
 for all $\sigma\in\Delta(n)$. And consequently,
 $$S_k(\Delta)\subseteq\bigcap_{\sigma\in\Delta(n)}S_k(\sigma).$$
\end{enumerate}
\end{proof}

\subsection{The vector space $N_{I}^k(\Delta)$ and $N_{I}^k(\sigma)$}
\subsubsection{The vector space $N_{I}^k(\Delta)$}\label{VIk-Delta-sect}
Let $X_\Delta$ be a smooth compact toric variety of dimension $n$.
For any  $I\in S_{\Delta}$, since $S_{\Delta}=\bigcup_{0\leq i\leq n}S(\Delta,i)$,
there exists a unique integer $0\leq i\leq n$, such that $I\in S(\Delta,i)$. 
We denote by $|I_\Delta|$ the integer $i$. Or in other words, 
\begin{equation}\label{|I|-eqn}
 |I_\Delta|=i~~\text{if and only if}~~ I\in S(\Delta,i).
\end{equation}

Suppose that $I\in int(F_I(\Delta))\cap M$. Then by Lemma \ref{SDelta-lem},
we have $F_I(\Delta)\in P_\Delta(n-i)$. 
Suppose that $$E_I(\Delta)=\{e(\alpha_{s_1}),\ldots e(\alpha_{s_l})\},$$ 
where $1\leq s_1<\ldots<s_l\leq r$.
By Lemma \ref{SDelta-lem}, 
$F^{\perp}_I(\Delta)=\sum_{1\leq t\leq l} \R\cdot e(\alpha_{s_t})$
 is a $i$-dimensional subspace of $N_\R$.
No loss of generality, suppose that 
$\{e(\alpha_{s_1}),\ldots,e(\alpha_{s_{i}})\}\subseteq\{e(\alpha_{s_1}),\ldots,e(\alpha_{s_{l}})\}$ 
is a basis of $F^{\perp}_I(\Delta)$, where $i\leq l$.

Let
 \begin{gather*}
 \cE_I(\Delta)=e(\alpha_{s_1}) \wedge\ldots\wedge e(\alpha_{s_i})\in\wedge^{|I_\Delta|} N.
\end{gather*}
Since $\{e(\alpha_{s_1}),\ldots,e(\alpha_{s_i})\}$ is a basis of $F^{\perp}_I(\Delta)$, 
we have
\begin{equation*} 
\wedge^i F^{\perp}_I(\Delta)=\R\cdot\cE_I(\Delta).
\end{equation*}
Thus  $\cE_I(\Delta)$ is well defined up to a scalar if we  choose different  basis 
$\{e(\alpha_{s_1}),\ldots,e(\alpha_{s_i})\}\subseteq E_I(\Delta)$ for $F^{\perp}_I(\Delta)$.

Let $N_I^k(\Delta)$ be a subspace of $\wedge^k N_{\C}$ ($0\leq k\leq n$) defined by
\begin{equation*}
N_I^k(\Delta)=
\begin{cases}
\C\cdot\mathcal{E}_I(\Delta)\wedge  (\wedge^{k-|I_\Delta|} N_\C)\quad 
&\text{for}~ 0<|I_\Delta|\leq k,\\
\wedge^k N_{\C}\quad &\text{for}~|I_\Delta|=0,\\
0\quad &\text{for}~|I_\Delta|>k.
\end{cases}
\end{equation*}

The following lemma gives a description of the vector space $N_I^k(\Delta)$.
\begin{lemma}\label{WDelta-lem}
Let $X_{\Delta}$ be a smooth compact toric variety.
For any $I\in S_\Delta$ and $x\in\wedge^k N_\C$, 
the following statements are equivalent:
\begin{enumerate}
\item 
$x\in N_I^k(\Delta)$,
\item $x\wedge e(\alpha_t)=0$ for all $e(\alpha_t)\in E_I(\Delta)$,
\item $x\wedge y=0$ for all $y\in F^{\perp}_I(\Delta)$.
\end{enumerate}
\end{lemma}
\begin{proof}
\begin{enumerate}
\item [(a)]
For $I\in S_\Delta$, in the case of $|I_\Delta|\leq k$, we suppose that $0\leq |I_\Delta|=i\leq k$ and 
 \begin{equation*}
 \cE_I(\Delta)=e(\alpha_{s_1}) \wedge\ldots\wedge e(\alpha_{s_i})\in\wedge^{i} N,
\end{equation*}
where $\{e(\alpha_{s_1}),\ldots,e(\alpha_{s_i})\}$ is a basis of 
$F^{\perp}_I(\Delta)\subseteq N_\C$.

We extend $\{e(\alpha_{s_1}),\ldots,e(\alpha_{s_i})\}$ to be a basis to $N_\C$,
\i.e.,  suppose that $$\{e(\alpha_{s_1}),\ldots,e(\alpha_{s_i}), f_1,\ldots, f_{n-i}\}$$
is a basis of $N_\C$. 
By some computations under the basis, we can prove Lemma \ref{WDelta-lem}. 
Here we omit the detail.
\item [(b)]
For $I\in S_\Delta$, In the case of $|I_\Delta|>k$, we have $N_I^k(\Delta)=0$. 
Lemma \ref{WDelta-lem} can be proved similarly.
\end{enumerate}
\end{proof}

\subsubsection{The vector space $N_{I}^k(\sigma)$}
Let $\sigma=\sum_{t=1}^{n}\R_{\geq 0}e_t(\sigma)$ be a smooth cone 
 of dimension $n$, where $\{e_1(\sigma),\ldots, e_n(\sigma)\}$ is a $\Z$-basis of $N$.
 Let $U_{\sigma}\cong\C^n$ be the affine toric variety associated the cone $\sigma$.
 Then we have $U_{\sigma}\cong\C^n$.
  
For any  $I\in S_{\sigma}$, since $S_{\sigma}=\bigcup_{0\leq j\leq n}S(\sigma,j)$,
there exists a unique integer $0\leq j\leq n$, such that $I\in S(\sigma,j)$. 
We denote by $|I_\sigma|$ the integer $j$, \i.e., $|I_\sigma|=j$.

Suppose that 
$$E_I(\sigma)=\{e_{s_1}(\sigma),\ldots, e_{s_j}(\sigma)\},$$
where $1\leq s_1<\ldots<s_j\leq n$.

Let
 \begin{equation*}
 \cE_I(\sigma)=e_{s_1}(\sigma) \wedge\ldots\wedge e_{s_j}(\sigma)\in\wedge^{|I_\sigma|} N.
\end{equation*}

Let $N_I^k(\sigma)$ be a subspace of $\wedge^k N_\C$ defined by
\begin{equation*}
N_I^k(\sigma)=
\begin{cases}
\C\cdot\mathcal{E}_I(\sigma)\wedge(\wedge^{k-|I_\sigma|} N_\C)
\quad&\text{for}~ 0<|I_\sigma|\leq k,\\
\wedge^k N_{\C}\quad&\text{for}~ |I_\sigma|=0,\\
0\quad&\text{for}~|I_\sigma|>k,
\end{cases}
\end{equation*}
for all $0\leq k\leq n$.

Similar to Lemma \ref{WDelta-lem}, we have
\begin{lemma}\label{Wsigma-lem}
Let $\sigma$ be a smooth cone of dimension $n$ in $N_\R$.
For any $I\in S_\sigma$ and $x\in\wedge^k N_\C$, 
the following statements are equivalent:
\begin{enumerate}
\item
$x\in N_I^k(\sigma)$,
\item
$x\wedge e_t(\sigma)=0$ for all $e_t(\sigma)\in E_I(\sigma)$,
\item 
$x\wedge y=0$ for all $y\in F^{\perp}_I(\sigma)$.
\end{enumerate}
\end{lemma}


\begin{proposition} \label{NIk-lem}
Let $X_{\Delta}$ be a smooth compact toric variety associated with a fan $\Delta$ in $N_{\R}$.
For all $I\in S_\Delta=\bigcap_{\sigma\in\Delta(n)}S_\sigma$,
we have
\begin{equation}\label{capNIk-eqn}
\bigcap_{\sigma\in\Delta(n)}N_I^k(\sigma)=N_I^k(\Delta).
\end{equation} 
Especially, we have
\begin{equation*}
\bigcap_{\sigma\in\Delta(n)}N_I^k(\sigma)=0
\quad\text{for all}~ I\in S_\Delta\backslash S_k(\Delta).
\end{equation*}
\end{proposition}

\begin{proof}
\begin{enumerate}
\item
Since $|\Delta|=\bigcup_{\sigma\in\Delta(n)}\sigma$, we have 
\begin{equation*}
E_I(\Delta)=\bigcup_{\sigma\in\Delta(n)} E_I(\sigma).
\end{equation*}
By Lemma \ref{Wsigma-lem}, $x\in N_I^k(\sigma)$ if and only if 
$x\wedge e(\alpha_t)=0$ for all $e(\alpha_t)\in E_I(\sigma)$.
 As a consequence, 
$$x\in\bigcap_{\sigma\in\Delta(n)}N_I^k(\sigma)$$
 if and only if
$$x\wedge e(\alpha_t)=0$$
 for all $e(\alpha_t)\in \bigcup_{\sigma\in\Delta(n)}E_I(\sigma)=E_I(\Delta)$.
 By Lemma \ref{WDelta-lem},  
 $$x\wedge e(\alpha_t)=0$$ for all $e(\alpha_t)\in E_I(\Delta)$ if and only if 
 $$x\in N_I^k(\Delta).$$
 Therefore, we have
 \begin{equation}\label{NIk-lem-eqn1}
 \bigcap_{\sigma\in\Delta(n)}N_I^k(\sigma)=N_I^k(\Delta).
 \end{equation}
 \item
 For any $I\in S_\Delta\backslash S_k(\Delta)$, we have $|I_\Delta|>k$, 
 which implies $N_I^k(\Delta)=0$. By Equation \eqref{NIk-lem-eqn1},
 we get $$\bigcap_{\sigma\in\Delta(n)}N_I^k(\sigma)=0$$
  for all $I\in S_\Delta\backslash S_k(\Delta)$.
  \end{enumerate}
\end{proof}

\begin{remark}
By Proposition \ref{NIk-lem},
$N_I^k(\Delta)$ is a subspace of $N_I^k(\sigma)$ for any $\sigma\in\Delta(n)$.
\end{remark}

\section{Holomorphic polyvector fields on toric varieties}\label{sect-multi}

\subsection{ The weight space $V_I^k$ for smooth toric varieties}
\begin{lemma}\label{VIk-lem}
Let $X$ be a smooth toric variety.
Let $V_I^k$ be the weight space corresponding to the character $I\in M$ for the 
$T$-action on $H^0(X,\wedge^k T_X)$.
\begin{enumerate}
\item We have $V_0^k=W^k$, where $V_0^k$ is the vector space consisting of all $T$-invariant holomorphic $k$-vector fields on $X$.
\item
For any holomorphic $k$-vector field $v\in V_{-I}^k$, there exists a unique 
 $k$-vector field $w\in W^k$, such that
\begin{equation}\label{VIKlem-eqn}
v|_T=\chi^I\cdot w|_T,
\end{equation}
where $v|_T$ and $\chi^I\cdot w|_T$ are the restrictions of $v$ and $\chi^I\cdot w$ on $T\subseteq X$.
\end{enumerate}
\end{lemma}
\begin{proof}
\begin{enumerate}
\item
 
For any holomorphic $T$-invariant $k$-vector field $v\in V_0^k$, the restriction of $v$ 
on $T\subseteq X$ is also $T$-invariant. 
Since the $T$-invariant $k$-vector fields on $T$ can be identified as the elements in 
$\wedge^k\Lie(T)\cong\wedge^N_{\C}$, there exists an element 
$x\in\wedge^k N_{\C}$, such that  
\begin{equation}
v|_{T}=\rho(x)|_{T}.
\end{equation}
Since the $k$-vector fields $v$ and $\rho(x)$ are holomorphic on $X$, 
and $T$ is dense in $X$, 
we have $$v=\rho(x)\in W^k$$ on $X$. 
Hence we obtain $V_0^k\subseteq W$.

On the other hand, any $w\in W^k$ is a $T$-invariant holomorphic $k$-vector field 
on $X$. Hence we have $W^k\subseteq V_0^k$.

By the arguments above, we get $V_0^k=W^k$ .

\item 
For any holomorphic $k$-vector field $v\in V_{-I}^k$, we have 
\begin{equation}\label{VIklem-eqn1}
 t_{*}(v)=\chi^{-I}(t)\cdot v
\end{equation}
for all $t\in T$, where $t_{*}(v)$ is the induced action of $t\in T$ 
on the $k$-vector field $v$.

Next we prove that $\chi^{-I}\cdot v|_T$ is a $T$-invariant holomorphic $k$-vector field 
on $T$. 
For any $t\in T$ and any point $p\in T\subseteq X$, we have
\begin{equation}\label{VIklem-eqn2}
(t_{*}(\chi^{-I}\cdot v))(p)=\chi^{-I}(t^{-1}\cdot p)\cdot (t_{*}(v))(p).
\end{equation}
By Equation \eqref{VIklem-eqn1}, we have
\begin{equation*}
( t_{*}(v))(p)=\chi^{-I}(t)\cdot v(p).
\end{equation*}
On the other hand, we have
\begin{equation*}
\chi^{-I}(t^{-1}\cdot p)=\chi^{-I}(t^{-1})\chi^{-I}(p)=\chi^{I}(t)\chi^{-I}(p).
\end{equation*}
Thus Equation \eqref{VIklem-eqn2} shows us
$$t_{*}(\chi^{-I}\cdot v)(p)=\chi^{-I}(p)\cdot v(p)=(\chi^{-I}\cdot v)(p)$$
for all $t\in T$ and $p\in T\subseteq X$.
Hence $\chi^{-I}\cdot v|_T$ is a $T$-invariant holomorphic $k$-vector field on $T$.

As a consequence, there exists $w\in W^k$ such that 
$$w|_T=\chi^{-I}\cdot v|_T,$$
or equivalently, $$v|_T=\chi^I\cdot w|_T.$$
\end{enumerate}
\end{proof}

By Lemma \ref{VIk-lem}, any holomorphic $k$-vector fields $v\in V_{-I}^k$ can be written as $v=\chi^I\cdot w$ on $T\subseteq X$, where $I\in M$ and $w\in W^k$.
 In general, $v=\chi^I\cdot w$ is a meromorphic $k$-vector field on $X$ for $I\in M$ and $w\in W^k$. If $v=\chi^I\cdot w$ has moveable singularity on $X$, by abuse of notations,
 we also use $v=\chi^I\cdot w$ to represent the corresponding holomorphic $k$-vector field on $X$,  and we say that $v=\chi^I\cdot w$ is holomorphic on $X$ in this paper.
 
 We have the following lemma.
\begin{lemma}\label{VIk-lem1}
Let $X$ be a smooth toric variety.  Let $v=\chi^I\cdot w$ be a meromorphic $k$-vector field on $X$, where $I\in M$ and $w\in W^k$. Then $v=\chi^I\cdot w\in V_{-I}^k$ 
if and only if $v=\chi^I\cdot w$ is holomorphic on $X$.
\end{lemma}
\begin{proof}
\begin{enumerate}
\item
$``\Longleftarrow"$ \quad For any $w\in W^k$ and $v=\chi^I\cdot w$, 
suppose that $v$ is holomorphic on $X$. We will show that $v\in V_{-I}^k$.

By the similar proof as in Lemma \ref{VIk-lem}, we get
\begin{equation}\label{VIklem-eqn3}
t_*(v)|_T=t_*(\chi^I\cdot w)|_T=\chi^{-I}(t)(\chi^I\cdot w)|_T=\chi^{-I}(t)\cdot v|_T
\end{equation}
for all $t\in T$.
If $v=\chi^I\cdot w$ is holomorphic on $X$, 
then $t_*(v)$ and $\chi^{-I}(t)\cdot v$ are both holomorphic on $X$.
Since $T$ is dense in $X$,  the Equation \eqref{VIklem-eqn3} implies that
\begin{equation}
t_*(v)=\chi^{-I}(t)\cdot v
\end{equation}
for all $t\in T$. Therefore we have $v\in V_{-I}^k$.

\item
$``\Longrightarrow"$ \quad If 
$v=\chi^I\cdot w\in V_{-I}^k\subseteq H^0(X,\wedge^k T_X)$, 
then $v$ is necessarily holomorphic on $X$.
\end{enumerate}
\end{proof}

\subsection{Holomorphic polyvector fields on $U_\sigma$}
Let $\sigma$ be a smooth cone of dimension $n$ in $N_\R$.
Let $U_{\sigma}\cong\C^n$ be the affine variety associated with $\sigma$. 

The differential forms on affine toric varieties has been studied by Danilov 
in \cite{Danilov}. 
Here we discuss the holomorphic polyvector fields on $U_\sigma$. 
The methods we use are essentially similar with \cite{Danilov}.

Let 
\begin{equation*}
W(\sigma)=\{w|_{U_\sigma}\mid w\in W\},
\end{equation*}
where $w|_{U_\sigma}$ is the restriction of the holomorphic vector field $w$ 
on $U_\sigma$.
Let
\begin{equation*}
W^k(\sigma)=\{w|_{U_\sigma}\mid w\in W^k\}.
\end{equation*}
Then $W^k(\sigma)$ is a subspace of $H^0(U_\sigma,\wedge^kT_{U_\sigma})$.
And we have $W^k(\sigma)=\wedge^k W(\sigma)$ $(0\leq k\leq n)$.

For any $I\in S_\sigma$, let
\begin{equation*}
\mathcal{V}_I(\sigma)=\rho(\cE_I(\sigma))|_{U_\sigma}\in W^{|I_\sigma|}(\sigma).
\end{equation*}
And let
\begin{equation}\label{WIkSigma-eqn}
W_I^k(\sigma)=\rho(N_I^k(\sigma))|_{U_\sigma}=
\begin{cases}
\C\cdot\mathcal{V}_I(\sigma)\wedge  W^{k-|I_\sigma|}(\sigma)
\quad&\text{for}~ 0<|I_\sigma|\leq k,\\
\wedge^k N_{\C}\quad&\text{for}~ |I_\sigma|=0,\\
0\quad &\text{for}~|I_\sigma|>k,
\end{cases}
\end{equation}
for all $0\leq k\leq n$. 
Then $W_I^k(\sigma)$ is a subspace of $W^k(\sigma)$. 

Suppose that 
$$\sigma=\sum_{t=1}^{n} \R_{\geq 0}\cdot e_t(\sigma),$$ 
where $\{e_1(\sigma),\ldots, e_n(\sigma)\}$ is a $\Z$-basis of $N$. 
Let $\{e_1^*(\sigma),\ldots, e_n^*(\sigma)\}\subset M$ be the dual basis of 
$\{e_1(\sigma),\ldots, e_n(\sigma)\}$,
and let $z_t=\chi^{e_t^*(\sigma)}$ $(1\leq t\leq n)$ be the affine coordinates on $U_\sigma\cong\C^n$. Then we have 
$$\rho(e_t(\sigma))=z_t\frac{\partial}{\partial z_t}$$
for $1\leq t\leq n$. Let $v_t=\rho(e_t(\sigma))$, $1\leq t\leq n$. Then we have
$$W(\sigma)=\{\sum_{t=1}^n a_t v_t \mid a_t\in\C\}.$$
Suppose $I=\sum_{t=1}^{n}m_t e_{t}^{*}\in M$. Then 
$$\chi^I=z_1^{m_1}\ldots z_n^{m_n}.$$
Let $$E_I(\sigma)=\{ e_{s_l}(\sigma)\mid m_{s_l}=-1,~1\leq l\leq r, ~1\leq s_1<\ldots <s_j\leq n\}.$$
Then we have $|I_{\sigma}|=j$
and 
$$\mathcal{V}_I(\sigma)=v_{s_1}\wedge\ldots\wedge v_{s_j}.$$
Moreover, we have
$$S_{\sigma}=\{I=\sum_{i=1}^{n}m_ie_i^*\in M\mid m_i\geq-1\}$$
and 
$$S_k(\sigma)=\{I=\sum_{i=1}^{n}m_ie_i^*\in M\mid m_i\geq-1, |I_\sigma|\leq k\}.$$

By Lemma $5.1$ in \cite{Hong 19}, we have
\begin{lemma}\label{VIksigma-lem1}
Let $\sigma$ be a smooth cone of dimension $n$ in $N_\R$.
Let $v=\chi^I\cdot w$ be a $k$-vector field on $U_\sigma$,
where $I\in M$ and $w\in W^k(\sigma)~ (w\neq 0)$. 
Then $v=\chi^I\cdot w$ is holomorphic on $U_\sigma$ 
if and only if 
\begin{equation*}
I\in S_k(\sigma)\quad\text{and}\quad w\in W_I^k(\sigma).
\end{equation*}
\end{lemma}

Let
\begin{equation*}
V_I^k(\sigma)=\chi^I\cdot W_I^k(\sigma)=\{\chi^I\cdot w\mid w\in W^k(\sigma)\}.
\end{equation*}
By Lemma \ref{VIksigma-lem1}, for any $I\in S_k(\sigma)$,
elements in $V_I^k(\sigma)$ are holomorphic $k$-vector fields on $X$.
Hence $V_I^k(\sigma)$ can be considered a subspace of $H^0(U_\sigma,\wedge^kT_{U_\sigma})$ for any $I\in S_k(\sigma)$.

For a smooth cone $\sigma$ of dimension $n$ in $N_{\R}$, 
the $T$-action on $U_\sigma$ induces a $T$-action on 
$H^0(U_\sigma,\wedge^kT_{U_\sigma})$.
We denote by $V_I^k(U_\sigma)$ the weight space corresponding to the character 
$I\in M$. 

The following lemma also comes from Lemma $5.1$ in \cite{Hong 19}.
\begin{lemma}\label{VIksigma-lem}
Let $\sigma$ be a smooth cone of dimension $n$ in $N_{\R}$. 
Let $U_\sigma\cong\C^n$ be the affine toric variety associated with the cone $\sigma$.
Then we have
\begin{equation*}
V_{-I}^k(U_\sigma)=
\begin{cases}
V_I^k(\sigma)\quad &\text{for all}\quad I\in S_k(\sigma),\\
0\quad &\text{for all}\quad I\in M\backslash S_k(\sigma).
\end{cases}
\end{equation*}
\end{lemma}

\subsection{Holomorphic polyvector fields on $X_\Delta$}
Let $X=X_\Delta$ be a smooth compact toric variety associated with a fan $\Delta$ in $N_\R\cong\R^n$. Recall that $X_\Delta$ is smooth if and only if each cone 
$\sigma\subset\Delta$ is smooth, \i.e., $\sigma$ is generated by a subset of a basis of $N$. And $X_{\Delta}$ is compact if and only if $$|\Delta|=\bigcup_{\sigma\in\Delta}\sigma=N_{\R}.$$
Let $U_{\sigma}\subseteq X_{\Delta}$ be the affine variety associated with a cone 
$\sigma\in\Delta(n)$. Then we have $X=\bigcup_{\sigma\in\Delta(n)} U_{\sigma}$.
As $X$ is smooth, we have $U_\sigma\cong\C^n$ for all $\sigma\in\Delta(n)$.
Moreover, we have $T\subset U_\sigma\subset X$ for all $\sigma\in\Delta(n)$.

\begin{lemma}\label{weightspace-lem1}
Let $X=X_\Delta$ be a compact smooth toric variety.
Let $v=\chi^I\cdot\rho(x)$ be a $k$-vector field on $X$,
where $I\in M$ and $x\in\wedge^k N_\C ~(x\neq 0)$. 
Then $v$ is holomorphic on $X$ if and only if 
\begin{equation*}
x\in N_I^k(\Delta) \quad \text{and}
\quad I\in\bigcap_{\sigma\in\Delta(n)}S_k(\sigma).
\end{equation*}
\end{lemma}
\begin{proof}
Since $X=\cup_{\sigma\in\Delta(n)}U_{\sigma}$, $v=\chi^I\cdot\rho(x)$ 
is holomorphic on $X$ if and only if $v=\chi^I\cdot\rho(x)$ is holomorphic on 
$U_\sigma$ for all $\sigma\in\Delta(n)$. 
By Lemma \ref{VIksigma-lem1},
$v=\chi^I\cdot\rho(x)$ is holomorphic on $U_\sigma$ if and only if
$I\in S_k(\sigma)$ and $\rho(x)\in W_I^k(\sigma)$. 
However, by Equation \eqref{WIkSigma-eqn}, we have
$$\rho(x)\in W_I^k(\sigma)\Longleftrightarrow x\in N_I^k(\sigma).$$
Thus $v$ is holomorphic on $X$ if and if 
\begin{equation*}
x\in \bigcap_{\sigma\in\Delta(n)}N_I^k(\sigma) \quad \text{and}
\quad I\in\bigcap_{\sigma\in\Delta(n)}S_k(\sigma).
\end{equation*}
\end{proof}

Let $X=X_\Delta$ be a smooth compact toric variety of dimension $n$. 
Let
\begin{equation*}
\mathcal{V}_I(\Delta)=\rho(\cE_I(\Delta))\in W^{|I_\Delta|}
\end{equation*}
for $I\in S_\Delta$.
Recall that
\begin{equation*}
W_I^k(\Delta)=\rho(N_I^k(\Delta))=
\begin{cases}
\C\cdot\mathcal{V}_I(\Delta)\wedge  W^{k-|I_\Delta|}\quad &\text{for}~0<|I_\Delta|\leq k,\\
W^k\quad &\text{for}~ |I_\Delta|=0,\\
0\quad &\text{for}~|I_\Delta|>k
\end{cases}
\end{equation*}
and
\begin{equation*}
V_I^k(\Delta)=\chi^I\cdot W_I^k(\Delta)=\{\chi^I\cdot w\mid w\in W_I^k(\Delta)\}.
\end{equation*}
The elements in $V_I^k(\Delta)$ are considered as meromorphic $k$-vector fields on $X=X_\Delta$.

\begin{lemma}\label{weightspace-lem2}
Let $X=X_\Delta$ be a smooth compact toric variety of dimension $n$. For all
$I\in S_k(\Delta)$, any $k$-vector field in $V_I^k(\Delta)$ is holomorphic on $X$.
\end{lemma}
\begin{proof}
For any $k$-vector field $v=\chi^I\cdot w\in V_I^k(\Delta)$, 
there exists $x\in N_I^k(\Delta)$, 
such that $w=\rho(x)$. By Proposition \ref{NIk-lem}, $x\in N_I^k(\Delta)$ implies
$x\in \bigcap_{\sigma\in\Delta(n)}N_I^k(\sigma)$.
By Proposition \ref{PFS-lem}, $I\in S_k(\Delta)$ implies
$I\in\bigcap_{\sigma\in\Delta(n)}S_k(\sigma)$.
By Lemma \ref{weightspace-lem1}, $v$ is holomorphic on $X$.
\end{proof}

By Lemma \ref{weightspace-lem2}, $V_I^k(\Delta)$ can be considered as a subspace of 
$H^0(X,\wedge^k T_X)$. 

\begin{proposition}\label{weightspace-lem} 
Let $X=X_\Delta$ be a smooth compact toric variety of dimension $n$. 
Let $V_{I}^k$ be the weight space corresponding to the character $I\in M$ for 
the $T$-action on $H^0(X,\wedge^k T_X)$. 
\begin{enumerate}
\item
We have
\begin{equation*}
V_{-I}^k=
\begin{cases}
V_I^k(\Delta)\quad &\text{for all}\quad I\in S_k(\Delta),\\
0\quad &\text{for all}\quad I\in M\backslash S_k(\Delta).
\end{cases}
\end{equation*}
\item
For any $I\in S_k(\Delta)$, we have
\begin{equation*}
\dim V_I^k(\Delta)={{n-|I_\Delta|}\choose{k-|I_\Delta|}}.
\end{equation*}
\end{enumerate}
\end{proposition}

\begin{proof}
\begin{enumerate}
\item

\begin{enumerate}
\item 
We will prove here $V_{-I}^k=V_I^k(\Delta)$ for all $I\in S_k(\Delta)$.

By Lemma \ref{VIk-lem}, any element $v\in V_{-I}^k$ can be written as 
$$v=\chi^I\cdot w,$$ where $w\in W^k$.
And by Lemma \ref{VIk-lem1}, $v=\chi^I\cdot w\in V_{-I}^k$ if and only if 
$v=\chi^I\cdot w$ is holomorphic on $X$.

Suppose that $w=\rho(x)$, where $x\in\wedge^k N_\C$.
By Lemma \ref{weightspace-lem1},  $v$ is holomorphic on $X$ if and only if 
\begin{equation}\label{weightspace-lem-eqn1}
x\in \bigcap_{\sigma\in\Delta(n)}N_I^k(\sigma) \quad \text{and}
\quad I\in\bigcap_{\sigma\in\Delta(n)}S_k(\sigma).
\end{equation}

However, Proposition \ref{NIk-lem} tells us
$$\bigcap_{\sigma\in\Delta(n)}N_I^k(\sigma)=N_I^k(\Delta)$$
for all $I\in S_k(\Delta)$. 
And by Proposition \ref{PFS-lem}, $I\in S_k(\Delta)$ implies 
$$I\in\bigcap_{\sigma\in\Delta(n)}S_k(\sigma).$$

As a consequence, for $I\in S_k(\Delta)$, 
$v=\chi^I\cdot\rho(x)$ is holomorphic on $X$ if and only if 
$$x\in N_I^k(\Delta),$$
which is equivalent to 
$$v\in V_I^k(\Delta).$$

Hence we have $V_{-I}^k=V_I^k(\Delta)$ for all $I\in S_k(\Delta)$.

\item
We will prove here $V_{-I}^k=0$ for all $I\notin S_k(\Delta)$.

By Proposition \ref{PFS-lem}, we have
\begin{equation}\label{weightspace-lem-eqn2.2}
S_k(\Delta)\subseteq\bigcap_{\sigma\in\Delta(n)}S_k(\sigma).
\end{equation}
As a consequence, for $I\notin S_k(\Delta)$, we have
\begin{equation*} 
I\notin\bigcap_{\sigma\in\Delta(n)}S_k(\sigma)\quad\text{or}\quad
I\in\bigcap_{\sigma\in\Delta(n)}S_k(\sigma)\backslash S_k(\Delta).
\end{equation*}
By Lemma \ref{VIk-lem}, any $v\in V_{-I}^k$ can be written as $v=\chi^I\cdot w$, 
where $w\in W^k$.
By Lemma \ref{VIk-lem1}, $v=\chi^I\cdot w\in V_{-I}^k$ if and only if 
$v=\chi^I\cdot w$ is holomorphic on $X$.
Suppose that $w=\rho(x)$, where $x\in\wedge^k N_\C$. Then $v=\chi^I\cdot\rho(x)$.
According to Lemma \ref{weightspace-lem1}, the $k$-vector field
$v=\chi^I\cdot\rho(x)$ $(v\neq 0)$ is holomorphic on $X$ if and only if 
 \begin{equation}\label{weightspace-lem-eqn2}
x\in \bigcap_{\sigma\in\Delta(n)}N_I^k(\sigma)~(x\neq0)~\quad \text{and}
\quad I\in\bigcap_{\sigma\in\Delta(n)}S_k(\sigma).
\end{equation}

\begin{enumerate}
\item 
For any $I\notin\bigcap_{\sigma\in\Delta(n)}S_k(\sigma)$, 
by Equation \eqref{weightspace-lem-eqn2}, we get
\begin{equation}\label{weightspace-lem-eqn2.1}
V_{-I}^k=0.
\end{equation}

\item
For any 
$I\in(\bigcap_{\sigma\in\Delta(n)}S_k(\sigma))\backslash S_k(\Delta)$, 
by Proposition \ref{PFS-lem}, we have
\begin{equation}\label{weightspace-lem-eqn3}
I\in(\bigcap_{\sigma\in\Delta(n)}S_k(\sigma))\backslash S_k(\Delta)\subseteq (\bigcap_{\sigma\in\Delta(n)}S_\sigma)\backslash S_k(\Delta)=S_\Delta\backslash S_k(\Delta).
\end{equation}
Moreover, by Proposition \ref{NIk-lem} and Equation \eqref{weightspace-lem-eqn3},  
we have
\begin{equation}\label{weightspace-lem-eqn4}
\bigcap_{\sigma\in\Delta(n)}N_I^k(\sigma)=0
\end{equation}
for all $I\in(\bigcap_{\sigma\in\Delta(n)}S_k(\sigma))\backslash S_k(\Delta)$.
By Equations \eqref{weightspace-lem-eqn2} and \eqref{weightspace-lem-eqn4}, 
we get
\begin{equation}\label{weightspace-lem-eqn5}
V_{-I}^k=0\quad\text{for all}\quad
I\in(\bigcap_{\sigma\in\Delta(n)}S_k(\sigma))\backslash S_k(\Delta).
\end{equation}
\end{enumerate}

By Equations \eqref{weightspace-lem-eqn2.2},  \eqref{weightspace-lem-eqn2.1} and  \eqref{weightspace-lem-eqn5}, 
we obtain $V_{-I}^k=0$ for $I\notin S_k(\Delta)$.

\end{enumerate}

\item
For any  $I\in S_k(\Delta)$, we have $0\leq |I_\Delta|\leq k$ . 
Suppose $0\leq |I_\Delta|=i\leq k$.
And suppose that
 \begin{gather*}
 \cE_I(\Delta)=e(\alpha_{s_1}) \wedge\ldots\wedge e(\alpha_{s_i})\in\wedge^{i} N,
\end{gather*}
where $\{e(\alpha_{s_1}),\ldots,e(\alpha_{s_i})\}$ is a basis of 
$F^{\perp}_I(\Delta)\subseteq N_\C$.

We extend $\{e(\alpha_{s_1}),\ldots,e(\alpha_{s_i})\}$ to be a basis to $N_\C$.
 Suppose the set
 $$\{e(\alpha_{s_1}),\ldots,e(\alpha_{s_i}), f_1,\ldots, f_{n-i}\}$$
is a basis of $N_\C$. Then the set
$$\{\rho(e(\alpha_{s_1})),\ldots,\rho(e(\alpha_{s_i})),\rho( f_1),\ldots, \rho(f_{n-i})\}$$
is a basis of $W$.
As a consequence, the set
$$\{\chi^I\cdot\rho(\cE_I(\Delta))\wedge\rho(f_{t_1})\wedge\ldots\rho(f_{t_{k-i}})\mid
1\leq t_1< \ldots<t_{k-i}\leq n-i \}$$
is a basis of the vector space 
$V_I^k(\Delta)=\chi^I\cdot\rho(\cE_I(\Delta))\wedge W^{k-i}$.
 Hence we have
$$\dim V_I^k(\Delta)={{n-i}\choose{k-i}}={{n-|I_\Delta|}\choose{k-|I_\Delta|}}.$$
\end{enumerate}
\end{proof}

\subsection{The proof of  Theorem \nameref{General-multiVect-thm}}
\begin{proof}
\begin{enumerate}
\item
By Equation \eqref{VIk-eqn}, we have
\begin{equation}\label{General-multiVect-thmpf-eqn1}
H^0(X, \wedge^kT_{X})=\bigoplus_{I\in M}V_I^k=\bigoplus_{I\in M}V_{-I}^k
\end{equation}
By Proposition \ref{weightspace-lem}, we have
\begin{equation}\label{General-multiVect-thmpf-eqn2}
V_{-I}^k=
\begin{cases}
V_I^k(\Delta)\quad &\text{for}~ I\in S_k(\Delta),\\
0\quad &\text{for}~I\notin S_k(\Delta).
\end{cases}
\end{equation}
By Equations \eqref{General-multiVect-thmpf-eqn1} and 
\eqref{General-multiVect-thmpf-eqn2}, we get
\begin{equation}\label{H0k-thmpf-eqn}
H^0(X, \wedge^kT_{X})=\bigoplus_{I\in S_k(\Delta)}V_I^k(\Delta).
\end{equation} 
\item
For any $I\in S(\Delta, i)$, $|I_\Delta|=i$, by Proposition \ref{weightspace-lem}, we have
\begin{equation}\label{dimVIk-thmpf-eqn}
\dim V_I^k(\Delta)={{n-i}\choose{k-i}}.
\end{equation}
According to the definition of $S(\Delta,i)$, we have
\begin{equation}\label{countSD-eqn}
\#S(\Delta,i)=\sum_{F_j\in P_\Delta(n-i)}\# (int(F_j)\cap M).
\end{equation}
By Equations \eqref{dimVIk-thmpf-eqn} and \eqref{countSD-eqn}, we have
\begin{equation}\label{dimVIKSD-eqn}
\sum_{I\in S(\Delta,i)}\dim V_I^k(\Delta)=
\sum_{F_j\in P_{\Delta}( n-i)}{{n-i}\choose{k-i}}\cdot\#(int(F_j)\cap M).
\end{equation}
Since $S_k(\Delta)=\bigcup_{i=0}^k S(\Delta, i),$ by Equation \eqref{H0k-thmpf-eqn}
and Equation \eqref{dimVIKSD-eqn}, we get
\begin{equation*}
\dim H^0(X, \wedge^kT_{X})=\sum_{i=0}^k \sum_{I\in S(\Delta,i)}\dim V_I^k(\Delta)
=\sum_{i=0}^k \sum_{F_j\in P_{\Delta}( n-i)}{{n-i}\choose{k-i}}\cdot\#(int(F_j)\cap M).
\end{equation*}
\end{enumerate}
\end{proof}

\end{document}